\documentclass[10pt]{amsart}

\usepackage{amsfonts,amsmath,amssymb,amsthm,amscd,amsxtra}
\usepackage{enumerate,epsfig,verbatim}

\usepackage{amssymb,amstext,amsmath,amscd,amsthm,amsfonts,enumerate,graphicx,latexsym, enumitem, enumerate}
\usepackage[usenames]{color}

\usepackage{nicefrac}
\usepackage{array}

\usepackage{comment}
\newcommand{\comm}[1]{{\color[rgb]{0.0, 0.5, 0.0} #1}}

\newcommand{\new}[1]{{\color{blue} #1}}

\usepackage{mathptmx}

\usepackage{amsfonts,amsmath,amssymb,amsthm,amscd,amsxtra}
\usepackage{enumerate,verbatim}
\usepackage{centernot}
\usepackage{todonotes}

\usepackage[T1]{fontenc}
\usepackage[usenames,dvipsnames]{pstricks}
\usepackage[mathscr]{eucal}
\usepackage[notcite,notref]{}
\usepackage{tikz}
\usetikzlibrary{matrix,quotes}

\usepackage[notcite, notref]{}

\usepackage[usenames,dvipsnames]{pstricks}
\usepackage[mathscr]{eucal}
\usepackage{amsfonts,amsmath,amssymb,amsthm,amscd,amsxtra}
\usepackage{enumerate,verbatim}
\usepackage[all,2cell,ps]{xy}
\usepackage{mathptmx}
\usepackage[notcite, notref]{}
\usepackage[pagebackref]{hyperref}
\usepackage{mathtools}
\usepackage{nicefrac}
\usepackage{array}
\usepackage{xcolor}
\usepackage[all,2cell,ps]{xy}
\usepackage{amsfonts}
\usepackage[margin=1.4in]{geometry}
\usepackage{color}
\usepackage[notcite,notref]{}
\usepackage{url}
\usepackage{datetime}
\usepackage{mathtools}
\usepackage[all,2cell,ps]{xy}
\usepackage{amssymb}
\usepackage{amsmath}
\usepackage{amsfonts}
\usepackage{amsmath}
\usepackage{amsthm}
\usepackage{amssymb}
\usepackage{amscd}
\usepackage{amsfonts}
\usepackage{amsxtra}     
\usepackage{epsfig}
\usepackage{verbatim}
\usepackage[all]{xypic}

\SelectTips{cm}{}

\def\latex/{{\protect\LaTeX}}
\def\latexe/{{\protect\LaTeXe}}
\def\amslatex/{{\protect\AmS-\protect\LaTeX}}
\def\tex/{{\protect\TeX}}
\def\amstex/{{\protect\AmS-\protect\TeX}}
\def\bibtex/{{Bib\protect\TeX}}
\def\makeindx/{\textit{MakeIndex}}

\usepackage{amsmath}
\usepackage{amsthm}
\usepackage{amssymb}
\usepackage{amscd}
\usepackage{amsxtra}     
\usepackage{epsfig}
\usepackage{verbatim}
\usepackage[all]{xypic}
\usepackage{enumerate}

\theoremstyle{plain} 

\newtheorem{thm}{Theorem}[section]

\newtheorem{prop}[thm]{Proposition}
\newtheorem{cor}[thm]{Corollary}

\theoremstyle{definition}
\newtheorem{chunk}[thm]{\hspace*{-1.065ex}\bf}
\newtheorem{lem}[thm]{Lemma}
\newtheorem{dfn}[thm]{Definition}

\newtheorem{eg}[thm]{Example}

\newtheorem{ques}[thm]{Question}

\newtheorem{rmk}[thm]{Remark}

\theoremstyle{remark}

\newtheorem*{claim*}{Claim}

\newcommand{\fa}{\mathfrak{a}}
\newcommand{\fb}{\mathfrak{b}}

\newcommand{\fn}{\mathfrak{n}}
\newcommand{\fm}{\mathfrak{m}}

\newcommand{\rG}{\mathrm{G}}

\newcommand{\CC}{\mathbb{C}}

\newcommand{\ZZ}{\mathbb{Z}}

\newcommand{\fp}{\mathfrak{p}}
\newcommand{\fq}{\mathfrak{q}}

\DeclareMathOperator{\K}{K}
 \DeclareMathOperator{\Tor}{Tor}
\DeclareMathOperator{\Ext}{Ext}

\DeclareMathOperator{\Hom}{Hom}

\DeclareMathOperator{\Tr}{\textnormal{Tr}}

\DeclareMathOperator{\len}{\textup{length}}

 \DeclareMathOperator{\Supp}{Supp}

 \DeclareMathOperator{\im}{im}

 \DeclareMathOperator{\pd}{pd}
 
 \DeclareMathOperator{\height}{height}
 \DeclareMathOperator{\length}{length}
 \DeclareMathOperator{\cx}{cx}

 \DeclareMathOperator{\depth}{depth}

 \DeclareMathOperator{\HH}{H}

 \DeclareMathOperator{\h}{\eta}

\newcommand{\ul}{\underline}

\def\Tr{\mathsf{Tr}\hspace{0.01in}}

\newcommand{\Ann}{\textup{Ann}}

\def\urltilda{\kern -.15em\lower .7ex\hbox{\~{}}\kern
  .04em}\def\urldot{\kern -.10em.\kern -.10em}\def\urlhttp{http\kern
  -.10em\lower -.1ex\hbox{:}\kern -.12em\lower 0ex\hbox{/}\kern
  -.18em\lower 0ex\hbox{/}}

\newcommand{\bb}{\left[ \begin{smallmatrix}}
\newcommand{\eb}{\end{smallmatrix} \right]}

\newcommand{\rV}{\mathrm{V}}
\begin{document}

\title[An extension of a depth inequality of Auslander]{An extension of a depth inequality of Auslander}


\author[O. Celikbas]{Olgur Celikbas}
\address{Olgur Celikbas\\
Department of Mathematics \\
West Virginia University\\
Morgantown, WV 26506-6310, U.S.A}
\email{olgur.celikbas@math.wvu.edu}

\author[U. Le]{Uyen Le}
\address{Uyen Le\\
Department of Mathematics \\
West Virginia University\\
Morgantown, WV 26506-6310, U.S.A}
\email{hle1@mix.wvu.edu}

\author[H. Matsui]{Hiroki Matsui} 
\address{Hiroki Matsui\\Graduate School of Mathematical Sciences\\ University of Tokyo, 3-8-1 Komaba, Meguro-ku, Tokyo 153-8914, Japan}
\email{mhiroki@ms.u-tokyo.ac.jp}

\thanks{Celikbas was partly supported by WVU Mathematics Excellence and Research Funds (MERF). Matsui was partly supported by JSPS Grant-in-Aid for JSPS Fellows 19J00158.}
\maketitle{}

\vspace{-0.2in}

\begin{abstract} In this paper, we consider a depth inequality of Auslander which holds for finitely generated Tor-rigid modules over commutative Noetherian local rings. We raise the question of whether such a depth inequality can be extended for $n$-Tor-rigid modules, and obtain an affirmative answer for 2-Tor-rigid modules that are generically free. Furthermore, in the appendix, we use Dao's eta function and determine new classes of Tor-rigid modules over hypersurfaces that are quotient of unramified regular local rings.
\end{abstract} 

\section{Introduction}

Throughout, $R$ denotes a commutative Noetherian local ring with unique maximal ideal $\fm$ and residue field $k$, and all $R$-modules are assumed to be finitely generated.

In this paper we are concerned with the following theorem of Auslander \cite{Au}, where $\depth_R(\fa, M)$ denotes the $\fa$-depth of $M$; see \ref{dcx1} and \ref{TR} for definitions and details.

\begin{thm} \label{thmint} (Auslander \cite{Au}) Let $R$ be a local ring, $M$ a nonzero $R$-module, and let $\fa$ be an ideal of $R$.
If $M$ is Tor-rigid, then it follows that $\depth_R(\fa, M) \leq \depth_R(\fa, R)$.
\end{thm}

We should note that the conclusion of Theorem \ref{thmint} holds over regular local rings due to the results of Auslander \cite{Au} and Lichtenbaum \cite{Li}; see \ref{TR}(i). In fact, the depth inequality considered in Theorem \ref{thmint} holds for modules of finite projective dimension over arbitrary local rings due to the new intersection theorem established by Roberts; see \cite[6.2.3, 13.4.1]{RB}. Furthermore, the aforementioned inequality holds for all modules over certain non-regular local rings including even dimensional simple singularities, see \ref{Te}(i). On the other hand such an inequality can fail in general, even over hypersurfaces:

\begin{eg} (\cite[2.5]{celikbas2018second}) \label{exint} Let $R=\CC[\![x,y,z,w]\!]/(xy)$, $M=R/(x)$, and let $\fa$ be the ideal of $R$ generated by $y$, $z$ and $w$. Then it follows $\depth_R(\fa,R)=2<3=\depth_R(\fa, M)$ so that $\depth_R(\fa, M)\leq \depth_R(\fa,R)+1$. Note $M\cong \Omega_R N$, where $N=R/(y)$, and $M$ and $N$ are not Tor-rigid since $\Tor_1^R(M,N)=0\neq \Tor_2^R(M,N)$.
\end{eg}


As all modules are $2$-Tor-rigid over a given hypersurface ring, motivated by Theorem \ref{thmint} and Example \ref{exint}, we raise the following question:

\begin{ques} \label{pytanie} Let $R$ be a local ring, $M$ be a nonzero $R$-module, and let $\fa$ be an ideal of $R$. Assume $M=\Omega^n_R N$ for some $n\geq 0$ and some $R$-module $N$ which is $(n+1)$-Tor-rigid. Then does it follow that $\depth_R(\fa, M) \leq \depth_R(\fa, R)+n$?
\end{ques}

Note that, due to Theorem \ref{thmint}, Question \ref{pytanie} is true in case $n=0$; see also \ref{TR}. The question is also true if $R$ is a complete intersection ring of codimension $c$ and $n$ equals $c$; see \ref{Tak2}. The main purpose of ths paper is to study Question \ref{pytanie} for the case where $n=1$. For that case we are able to obtain an affirmative answer to the question under mild conditions. More precisely, we prove: 

\begin{thm}\label{hy} Let $R$ be a local ring, $\fa$ be an ideal of $R$, and let $M$ be a nonzero $R$-module such that $M=\Omega_R N$ for some $R$-module $N$ which is $2$-Tor-rigid and generically free (e.g., $R$ is reduced). If $\depth_R(\fa,R)\geq 1$, then it follows that $\depth_R(\fa, M)\leq \depth_R(\fa,R)+1$.
\end{thm}

Let us note that Theorem \ref{hy} follows as a consequence of our main result, namely Theorem \ref{di}; see Corollary \ref{dicor}. Let us also note that Theorem \ref{di} exploits the notion of Tor-rigidity developed by Auslander, and establishes a depth inequality that is more general from the one stated in Theorem \ref{hy}. 

The key ingredient for the proof of Theorem \ref{di}, and hence for the proof of Theorem \ref{hy}, is Proposition \ref{keylem0} which yields the existence of a certain short exact sequence involving the syzygy modules. We should point out that Proposition \ref{keylem0} corroborates a result of Herzog and Popescu \cite[2.1]{HP} and of Takahashi \cite[2.2]{Tak10}, and it is proved at the end of section 4; see also Corollary \ref{propint}.


As our results rely upon Tor-rigidity, in the appendix, we use Dao's eta function and determine new classes of $n$-Tor-rigid modules over complete intersections that are quotient of unramified regular local rings.



\section{Preliminaries}

In this section we record several preliminary definitions and results that are used in the paper. 

\begin{chunk} Let $R$ be a ring and let $M$ and $N$ be $R$-modules. If $M\oplus F \cong N \oplus G$ for some free $R$-modules $F$ and $G$, then $M$ and $N$ are said to be stably isomorphic. As it does not affect our arguments, we do not separate isomorphic and stably isomorphic modules.
\end{chunk}


\begin{chunk} Let $R$ be a ring and let $M$ be an $R$-module. Given an integer $n\geq 1$, we denote by $\Omega^{n}_RM$ the $n$th \emph{syzygy} of $M$, namely, the image of the $n$-th differential map in a minimal free resolution of $M$. As a convention, we set $\Omega^{0}_RM=M$ and $\Omega^{1}_RM= \Omega_R M$. 

The \emph{transpose} $\Tr M$ of $M$ is the cokernel of $f^{\ast} = \Hom_R(f,R)$, where $F_1 \stackrel{f} {\longrightarrow} F_0 \to M \to 0$ is a part of a minimal free resolution of $M$; see \cite[12.3]{AuBr}. 

Note that the transpose and the syzygy of $M$ are uniquely determined up to isomorphism, since so is a minimal free resolution of $M$. 
\end{chunk}

\begin{chunk} \label{dcx1}  Let $R$ be a ring, $M$ be an $R$-module, and let $\fa$ be an ideal of $R$. If $\fa M\neq M$, then the \emph{$\fa$-depth} of $M$  (or the \emph{grade} of $\fa$ on $M$), denoted by $\depth_R(\fa, M)$, is defined to be the common length of maximal $M$-regular sequences in $\fa$; see \cite[1.2.6]{BH}. In case $\fa M =M$, then we set $\depth_R(\fa, M)=\infty$ (in particular, we have $\depth_R(0)=\infty$). Although we write $\depth_R(\fa, R)$ throughout the paper, we note that $\depth_R(\fa, R)$ is nothing but the height of the ideal $\fa$ in case $R$ is a Cohen-Macaulay ring. Furthermore, we set $\depth_R(M)= \depth_R(\fm, M)$. 

The following basic facts play an important role in the proofs of Proposition \ref{ann} and Theorem \ref{di}.
\begin{enumerate}[label=(\roman*)]
\item $\depth_R(\fa, M) = \inf \{\depth_{R_\fp}(M_\fp) \mid \fp \in \rV(\fa)\}$; see \cite[1.2.10(a)]{BH}.
\item $\depth_R(\fa, R)= \inf \{i \in \ZZ: \Ext^i_R(R/\fa,R)\neq 0\}$; see \cite[1.2.10(e)]{BH}.
\item If $\underline{x} \subseteq \fa$ is a regular sequence of length $n$ on $M$, then $\depth_R(\fa, M/\underline{x}M)= \depth_R(\fa, M)-n$; see \cite[1.2.10(d)]{BH}.
\pushQED{\qed} 
\qedhere 
\popQED
\end{enumerate}
\end{chunk}

\begin{chunk} Let $R$ be a ring, $M$ be an $R$-module, and let $n\geq 1$ be an integer. Then $M$ is said to satisfy $(\widetilde{S}_n)$ if $\depth_{R_\fp}(M_\fp) \ge \min \{n, \depth (R_\fp)\}$ for all $\fp \in \Supp_R(M)$. Note that, if $R$ is Cohen-Macaulay, then $M$ satisfies $(\widetilde{S}_n)$ if and only if $M$ satisfies \emph{Serre's condition} $(S_n)$; see, for example, \cite[page 3]{EG}. \pushQED{\qed} 
\qedhere 
\popQED
\end{chunk}

We make use of the following properties in the proof of Proposition \ref{keylem0} and Corollary \ref{dici}. Note that, if $n\geq 0$, then $ \widetilde{X}^{n}(R)$ denotes the set of all prime ideals $\fp$ of $R$ such that $\depth(R_{\fp})\leq n$.

\begin{chunk}\label{tf} Let $R$ be a ring, $M$ be a nonzero $R$-module, and let $n\geq 1$ be an integer.  
\begin{enumerate}[label=(\roman*)]
\item If $\Ext_R^i(M, R) = 0$ for all $i=1, \ldots, n$, then it follows that $\Omega_R^n \Tr \Omega_R^n M \cong \Tr M$ and so $\Tr M$ is an $n$th syzygy module; see \cite[2.17]{Au}.
\item If $M$ is an $n$th syzygy module, then $M$ satisfies $(\widetilde{S}_n)$ so that each $R$-regular sequence of length at most $n$ is also $M$-regular; see \cite[4.25]{Au} and \cite[Prop. 2]{Mal}.
\item If $M$ is locally free on $ \widetilde{X}^{n-1}(R)$ and $M$ satisfies $(\widetilde{S}_n)$, then it follows that $M=\Omega_R^n N$ for some $R$-module $N$, where $\Ext^i_R(N,R)=0$ for all $i=1, \ldots, n$; see \cite[2.17 and 4.25]{Au}. 
\end{enumerate}
\end{chunk}

\begin{chunk}  \label{cx} Let $R$ be a ring and let $M$ be an $R$-module. The \emph{complexity} $\cx_R(M)$ of $M$ is the smallest integer $r\geq 0$ such that the $n$th Betti number of $M$  is bounded by a polynomial in $n$ of degree $r-1$ for all $n\geq 0$; see \cite[3.1]{Av1}. 

It follows that $\cx_R(M)=0$ if and only if $\pd_R(M)<\infty$, and $\cx_R(M)\leq 1$ if and only if $M$ has bounded Betti numbers. Moreover, if $R$ is a complete intersection, then $\cx_R(M)$ cannot exceed the codimension of $R$; see, for example, \cite[5.6]{AGP}.
\end{chunk}

\begin{chunk} \label{TR} Let $R$ be a ring, $M$ be an $R$-module, and let $n\geq 1$ be an integer. Then $M$ is said to be {\it $n$-Tor-rigid} provided that the following condition holds:
if $\Tor_i^R(M, N) = 0$ for all $i= t+1, \ldots, t+n$ for some $R$-module $N$ and some integer $t\geq 0$, then it follows that $\Tor_i^R(M, N) = 0$ for all $i \ge t+1$. The $n=1$ case of this definition is known as the {\it Tor-rigidity} \cite{Au}: $M$ is said to be Tor-rigid if it is $1$-Tor-rigid. 

Tor-rigidity is a subtle property, but examples of such modules are abundant in the literature. Here we record a few examples and refer the reader to \cite{Daosurvey} for further details and examples.
\begin{enumerate}[label=(\roman*)]  
\item (\cite[2.2]{Au} and \cite[Cor. 1]{Li}) If $R$ is regular, then each $R$-module is Tor-rigid.
\item (\cite[2.4]{HW1} and \cite[Thm. 3]{Li}) If $R$ is a hypersurface, that is a quotient of an unramified regular local ring, then each $R$-module that has finite length, or has finite projective dimension, is Tor-rigid.
\item (\cite[1.6]{Mu}) If $R$ is a complete intersection of codimension $c$, then each $R$-module is $(c+1)$-Tor-rigid. Therefore, if $c=1$, then each $R$-module is 2-Tor-rigid. 
\item Let $R$ be a complete intersection ring of positive codimension $c$ such that $\widehat{R}=S/(\underline{x})$ for some unramified regular local ring $(S, \fn)$ and some $S$-regular sequence $\underline{x}\subseteq \fn^2$ of length $c$. Each $R$-module that has complexity strictly less than $c$ is $c$-Tor-rigid. Therefore, if $c=2$, then each $R$-module that has bounded Betti numbers is $2$-Tor-rigid; see \cite[6.8]{Da2}.
\item (\cite[Thm. 5(ii)]{B}) If $I$ is a Burch ideal of $R$, i.e.,  if $\fm I \neq \fm(I : \fm)$, then $R/I$ is $2$-Tor-rigid. 
\item (\cite[page 316]{LV}) If $M$ is nonzero such that $\depth_R(M)\ge 1$, then $\fm M$ is $2$-Tor-rigid.
\pushQED{\qed} 
\qedhere 
\popQED
\end{enumerate}
\end{chunk}

The key ingredient of our argument is the following result; it allows us to tackle the problem on hand by using the Tor-rigidity property; see \ref{TR}.

\begin{prop}\label{keylem0} Let $R$ be a local ring, $N$ a nonzero $R$-module, and let $M= \Omega^n_R N$ for some $n\geq 1$.
Assume there is an $R$-regular sequence $\underline{x}= x_1, \ldots, x_n$ of length $n$ such that $\underline{x} \cdot \Ext_R^1(N, \Omega_R N) = 0$.
Then there is a short exact sequence of $R$-modules
\begin{equation}\tag{\ref{keylem0}.1}
0 \longrightarrow F \longrightarrow \bigoplus_{i = 0}^n \bigg(\Omega_R^{i+n-1}N\bigg)^{\oplus\left(\begin{smallmatrix} n \\ i \\ \end{smallmatrix} \right)} \longrightarrow\Omega_R^{n-1}(M/\underline{x}M) \longrightarrow 0,
\end{equation}
where $F$ is free.\pushQED{\qed} 
\qedhere 
\popQED
\end{prop}

The proof of Proposition \ref{keylem0} is quite involved, and hence it is deferred to Section 4. Here we record an important consequence of the proposition which is used later in the sequel.

\begin{cor} \label{keylemcor} Let $R$ be a local ring, $N$ a nonzero $R$-module, and let $M= \Omega^n_R N$ for some $n\geq 1$.
Assume the following conditions hold:
\begin{enumerate}[label=(\roman*)]
\item $N$ is $(n+1)$-Tor-rigid.
\item $\underline{x} \cdot \Ext_R^1(N, \Omega_R N) = 0$ for some $R$-regular sequence $\underline{x}$ of length $n$.
\end{enumerate}
Then it follows that $\Omega^{n-1}_R(M/\underline{x}M)$ is Tor-rigid.
\end{cor}

\begin{proof} Note, since $N$ is $(n+1)$-Tor-rigid, it follows that $ \bigoplus_{i = 0}^n \big(\Omega_R^{i+n-1}N\big)^{\oplus\left(\begin{smallmatrix} n \\ i \\ \end{smallmatrix} \right)}$ is Tor-rigid; see \ref{TR}.
Therefore, we conclude by (\ref{keylem0}.1) that $\Omega_R^{n-1}(M/\underline{x}M)$ is Tor-rigid.
\end{proof}

\begin{prop}\label{ann} Let $R$ be a local ring, $M$ and $N$ be $R$-modules, $\fa$ be a proper ideal of $R$, and let $n\geq 1$. Assume the following conditions hold:
\begin{enumerate}[label=(\roman*)]
\item	 $M$ satisfies $(\widetilde{S}_n)$.
\item $\depth_R(\fa, R) \ge n$. 
\item $N$ is locally free on $\widetilde{X}^{n-1}(R)$.
\end{enumerate}
Then there is a sequence $\underline{x} \subseteq \fa$ of length $n$ such that $\underline{x} \cdot \Ext_R^1(N, \Omega_RN) = 0$, and $\underline{x}$ is both $R$ and $M$-regular.
\end{prop}

\begin{proof} We have, by assumption, that $\depth_R(\fa, R) = \inf \{\depth(R_\fp) \mid \fp \in \rV(\fa)\}\geq n$; see \ref{dcx1}(i). Hence, for each $\fq \in \rV(\fa)$, it follows that $\depth(R_\fq)\geq n$. 

Set $\fb = \Ann_R (\Ext_R^1(N, \Omega_RN))$. If $\fq \in \rV(\fb)$, then we have $\depth(R_\fq)\geq n$: otherwise, $\fq \in \widetilde{X}^{n-1}(R)$ and hence $\Ext_R^1(N, \Omega_RN)_{\fq}=0$ since 
$N$ is locally free on $\widetilde{X}^{n-1}(R)$. 
Therefore, if $\fq \in \rV(\fa) \cup \rV(\fb)$, then it follows that $\depth(R_\fq)\geq n$. Furthermore, if $\fq \in \rV(\fa) \cup \rV(\fb)$, then we have $\depth_{R_\fq}(M_\fq)\geq n$ since $M$ satisfies $(\widetilde{S}_n)$ and $\depth(R_\fq)\geq n$. Consequently, we use \ref{dcx1}(i) and \cite[1.2.10(c)]{BH}, and obtain:
\begin{equation}\tag{\ref{ann}.1}
\depth_R(\fa \cap \fb, M\oplus R) = \inf\{\depth_R(\fa, M\oplus R), \depth_R(\fb, M\oplus R)\} \geq n.
\end{equation}

Now, by using (\ref{ann}.1), we can choose a sequence $\underline{x} \subseteq \fa \cap \fb \subseteq \fa$ of length $n$, as claimed.
\end{proof}

The next result is known for the case where $r=0$; see, for example, \cite[3.4]{onex}.

\begin{lem} \label{goodlemma} Let $R$ be a local ring, $A$ and $B$ be $R$-modules with $A\neq 0$, and let $m\geq 1$, $r\geq 0$ be integers. Assume $\Tr \Omega_R^mB$ is an $r$th syzygy module. Assume further $\Omega_R^rA$ is Tor-rigid. If $\Ext^m_R(B,A)=0$, then it follows that $\Ext^m_R(B,R)=0$.
\end{lem}

\begin{proof} Assume $\Ext^m_R(B,A)=0$, and consider the four term exact sequence that follows from \cite[2.8(b)]{AuBr}:
\begin{align}\tag{\ref{goodlemma}.1}
\Tor_2^R(\Tr \Omega_R^mB, A) &\to \Ext_R^m(B, R) \otimes_R A \to \Ext_R^m(B, A) \to \Tor_1^R(\Tr \Omega_R^mB, A) \to 0.
\end{align}
Note that, as $\Ext_R^m(B, A)$ vanishes, so does $\Tor_1^R(\Tr \Omega_R^mB, A)$ by (\ref{goodlemma}.1). Also, due to the hypothesis, it follows that $\Tr \Omega_R^mB \cong \Omega_R^r X$ for some $R$-module $X$. So, since  $\Tor_1^R(\Tr \Omega_R^mB, A) \cong \Tor_1^R(X, \Omega_R^rA)$ and $\Omega_R^rA$ is Tor-rigid, we conclude that $\Tor_2^R(\Tr \Omega_R^mB, A)=0$. Hence, as $A\neq 0$, (\ref{goodlemma}.1)
implies that $\Ext^m_R(B,R)=0$.
\end{proof}


\section{Main result and its corollaries}

In this section we prove the main result of this paper, namely Theorem \ref{di}. Prior to that, we note that Question \ref{pytanie} is true in case the ring in question is a complete intersection of codimension $c$ and the integer $n$ considered equals $c-1$; this fact has been explained to us by Shunsuke Takagi.

\begin{chunk} \label{Tak1} Let $R$ be a ring such that $R=S/(\underline{x})$ for some local ring $(S, \fn)$ and some $S$-regular sequence $\underline{x}\subseteq \fn$ of length $c$. Assume the depth inequality $\depth_S(\fb, N) \leq \depth_S(\fb, S)$ holds for each ideal $\fb$ of $S$ and for each $S$-module $N$. Let $M$ be an $R$-module and let $\fa$ be an ideal of $R$. Then $\fa=\fb/(\underline{x})$ for some ideal $\fb$ of $S$. Now, it follows $\depth_R(\fa,M)=\depth_S(\fb,M)\leq \depth_S(\fb,S)=\depth_S(\fb, R)+c=\depth_R(\fa,R)+c$.\pushQED{\qed} 
\qedhere 
\popQED
\end{chunk}

Recall that each module is Tor-rigid over a regular local ring; see \ref{TR}(i). Therefore, we obtain:

\begin{chunk} \label{Tak2} Let $R$ be a complete intersection ring of codimension $c$, $M$ be an $R$-module, and let $\fa$ be an ideal of $R$. Then it follows from Theorem \ref{thmint} and \ref{Tak1} that 
$\depth_R(\fa,M)\leq\depth_R(\fa,R)+c$.\pushQED{\qed} 
\qedhere 
\popQED
\end{chunk}

Next we state and prove Theorem \ref{di}. We should note that the case where $n=0$ of the theorem is nothing but Theorem \ref{thmint}. In other words, Theorem \ref{di} yields an extension of Theorem \ref{thmint}.

\begin{thm}\label{di} Let $R$ be a local ring, $N$ be an $R$-module, and let $\fa$ be an ideal of $R$. Set $M=\Omega_R^n N$ for some integer $n\geq 0$ and $m=\depth_R(\fa, R)$.
Assume the following conditions hold:
\begin{enumerate}[label=(\roman*)]
\item $M\neq 0$ and $m \geq n$. 
\item $N$ is $(n+1)$-Tor-rigid.
\end{enumerate}
If $n\geq 1$, we further assume:
\begin{enumerate}[label=(\roman*), resume]
\item $N$ is locally free on $\widetilde{X}^{n-1}(R)$.
\item $\Tr\Omega_R^{m}(R/\fa)$ is an $(n-1)$st syzygy module.
\end{enumerate}
Then it follows that $\depth_R(\fa, M) \le m+n$.
\end{thm}

\begin{proof} Note that there is nothing to prove if $\fa=0$, or $\fa=R$, or $\depth_R(\fa, M) \le n$; see \ref{dcx1}. Note also that the case where $n=0$ follows from Theorem \ref{thmint}. Hence we may assume $\fa$ is a proper ideal and $\depth_R(\fa, M) > n\geq 1$.

As $M$ is an $n$th syzygy module, we see that $M$ satisfies $(\widetilde{S}_n)$; see \ref{tf}(ii). Therefore, since $N$ is locally free on $\widetilde{X}^{n-1}(R)$ and $\depth_R(\fa, R) \ge n$, it follows from  Proposition \ref{ann} that there exists a sequence $\underline{x} \subseteq \fa$ of length $n$ which is both $R$ and $M$-regular and $\underline{x} \cdot \Ext_R^1(N, \Omega_R(N)) = 0$. Now, as $N$ is $(n+1)$-Tor-rigid, Corollary \ref{keylemcor} shows that $\Omega_R^{n-1}(M/\underline{x} M)$ is Tor-rigid.

Let $h=\depth_R(\fa, M/\underline{x}M)$ and suppose $h>m$. Then it follows that $\Ext_R^m(R/\fa, M/\underline{x}M) = 0$; see \ref{dcx1}(ii).
Now, letting $A=M/\underline{x} M$, $B=R/\fa$ and $r=n-1$, we conclude from Lemma \ref{goodlemma} that $\Ext_R^m(R/\fa, R) = 0$. This yields a contradiction since $m=\depth_R(\fa, R)$; see \ref{dcx1}(ii).
Therefore, we have that $h\leq m$. This establishes the required inequality since $h=\depth_R(\fa, M)-n$; see \ref{dcx1}(iii).
\end{proof}

Next we proceed to obtain several consequences of Theorem \ref{di}. First we separate the case where $n=1$, which is nothing but Theorem \ref{hy} advertised in the introduction:

\begin{cor}\label{dicor} Let $R$ be a local ring, and let $\fa$ be an ideal of $R$ such that $\depth_R(\fa, R)\geq 1$. Set $M=\Omega_R N$ for some $R$-module $N$, where $N$ is $2$-Tor-rigid and generically free. If $M\neq 0$, then it follows that $\depth_R(\fa, M) \le \depth_R(\fa, R)+1$.\pushQED{\qed} 
\qedhere 
\popQED
\end{cor}



\begin{cor} \label{3.5} Let $R$ be a local complete intersection ring of codimension $c$ such that $\widehat{R}=S/(\underline{x})$ for some unramified regular ring $(S, \fn)$ and some $S$-regular sequence $\underline{x}\subseteq \fn^2$ of length $c$, where $c\leq 2$. Let $M$ be a nonzero $R$-module, and let $\fa$ be an ideal of $R$. Assume $M$ is generically free and torsion-free. Assume further $M$ has bounded Betti numbers. Then it follows that $\depth_R(\fa, M) \leq \depth_R(\fa, R)+1$.
\end{cor}

\begin{proof} 
Note that, as $R$ is Cohen-Macaulay, $M$ is generically free and torsion-free, we have that $M \cong \Omega_R N$ for some $R$-module $N$. Since $M$ has bounded Betti numbers, so does $N$. Hence it follows that $N$ is 2-Tor-rigid; see \ref{TR}(iv). Furthermore, $N$ is generically free because $M$ is generically free. Thus the result follows from Corollary \ref{dicor}.
\end{proof}




\begin{cor}\label{dicor2} Let $R$ be a local ring and let $\fa$ be an ideal of $R$ such that $\depth_R(\fa, R)\geq 1$. Let $N$ be a nonzero $R$-module such that $N$ is generically free and $\depth_R(N)\geq 1$.
If $M=\Omega_R(\fm N)\neq 0$, then it follows that $\depth_R(\fa, M) \le \depth_R(\fa, R)+1$.
\end{cor}

\begin{proof} Note that we may assume $R$ is not Artinian. Hence, $\fm N$ is generically free. Moreover, $\fm N$ is $2$-Tor-rigid; see \ref{TR}(iv). Therefore, the claim follows from Corollary \ref{dicor}.
\end{proof}


\begin{cor}\label{dicor3} Let $R$ be a local ring, $\fa$ be an ideal of $R$ and let $\fb$ is a Burch ideal of $R$. Assume $\depth_R(\fa, R)\geq 1$ and $\depth_R(\fb, R)\geq 1$. Then it follows that $\depth_R(\fa, \fb) \le \depth_R(\fa, R)+1$.
\end{cor}

\begin{proof} Note that $\fb=\Omega_R N$, where $N=R/\fb$ is $2$-Tor-rigid; see \ref{TR}(v). Moreover, $N$ is generically free since $\depth_R(\fb, R)\geq 1$; see \ref{dcx1}(i). Hence, the result follows from Corollary \ref{dicor}. 
\end{proof}

It is known that integrally closed ideals are Burch over local rings that have positive depth; see \cite[2.2 (3) and (4)]{BBB}. Therefore, Corollary \ref{dicor3} yields:

\begin{cor}\label{dicor4} Let $R$ be a local ring, and let $\fa$ and $\fb$ be ideals of $R$. Assume $\depth_R(\fa, R)\geq 1$ and $\depth_R(\fb, R)\geq 1$. If $\fb$ is integrally closed, then it follows that $\depth_R(\fa, \fb) \le \depth_R(\fa, R)+1$.\pushQED{\qed} 
\qedhere 
\popQED
\end{cor}




In the following corollaries, we show that condition (iv) of Theorem \ref{di} holds if $\fa$ is a Cohen-Macaulay ideal, i.e., $R/\fa$ is a Cohen-Macaulay ring.

\begin{cor}\label{diGor} Let $R$ be a Gorenstein local ring, $N$ be an $R$-module, and let $\fa$ be an ideal of $R$. Set $M=\Omega_R^n N$ for some integer $n\geq 1$ and $m=\depth_R(\fa, R)$.
Assume the following conditions hold:
\begin{enumerate}[label=(\roman*)]
\item $\fa$ is a Cohen-Macaulay ideal.
\item $M\neq 0$ and $m \geq n$. 
\item $N$ is locally free on $\widetilde{X}^{n-1}(R)$.
\item $N$ is $(n+1)$-Tor-rigid.
\end{enumerate}
Then it follows that $\depth_R(\fa, M) \le m+n$.
\end{cor}

\begin{proof} Note that, as $R/\fa$ is a Cohen-Macaulay ring, it follows $\depth(R/\fa)=\dim(R)-m$, and also $\Ext_R^i(R/ \fa, R)=0$ for $i\neq m$ by the local duality theorem; see \cite[3.5.8]{BH}. Therefore, $\Tr\Omega_R^{m}(R/\fa)$ is an $(n-1)$st syzygy module since $\Ext_R^i(R/ \fa, R) = 0$ for all $i=m +1, \ldots, m+n-1$; see \ref{tf}(i). Now, since all the hypotheses of Theorem \ref{di} hold, the required depth inequality follows from Theorem \ref{di}.
\end{proof}


The next corollary corroborates Corollary \ref{3.5}:

\begin{cor}\label{dici} Let $R$ be a local complete intersection ring of codimension $c$ such that $\widehat{R}=S/(\underline{x})$ for some unramified regular  ring $(S, \fn)$ and some $S$-regular sequence $\underline{x}\subseteq \fn^2$ of length $c$, where $c\geq 2$. Let $M$ be a nonzero $R$-module, and let $\fa$ be an ideal of $R$.
Assume the following hold:
\begin{enumerate}[label=(\roman*), resume]
\item $\fa$ is a Cohen-Macaulay ideal such that $\depth(\fa,R) \geq c-1$. 
\item $\cx_R(M)<c$.
\item $M$ satisfies $(\widetilde{S}_{c-1})$.
\item $M$ is locally free on $\widetilde{X}^{c-2}(R)$.
\end{enumerate}
Then it follows that $\depth_R(\fa, M) \leq  \depth_R(\fa,R)+c-1$.
\end{cor}

\begin{proof} Note that, by \ref{tf}(iii), we have $M=\Omega_R ^{c-1}N$ for some $R$-module $N$, where $\Ext_R^i(N, R)=0$ for all $i=1, \ldots, c-1$. Let $\fp \in \widetilde{X}^{c-2}(R)$. Then, since $M$ is locally free on $\widetilde{X}^{c-2}(R)$, it follows $\pd_{R_{\fp}}(N_{\fp})\leq c-1$. As $\Ext_R^i(N, R)=0$ for all $i=1, \ldots, c-1$, we conclude that $N_{\fp}$ is free. This shows that $N$ is locally free on $\widetilde{X}^{c-2}(R)$. Furthemore, as $\cx_R(N)=\cx_R(M)<c$, it follows that $N$ is $c$-Tor-rigid; see \ref{TR}(iv). Hence the result follows from Corollary \ref{diGor} by setting $n=c-1$.
\end{proof}

\begin{rmk} Let us note that, if $c=2$ in Corollary \ref{dici}, then the Cohen-Macaulay assumption on the ideal $\fa$ is not needed due to Corollary \ref{3.5}. Moreover, the assumption $\cx_R(M)<c$ in Corollary \ref{dici} implies the vanishing of the eta function if $R$ is an isolated singularity; in this case $M$ would be a $c$-Tor-rigid module; see \cite[6.3 and 6.8]{Da2}. In the appendix we recall the definition of the eta function and discuss some of its applications that are related to our results.
\end{rmk}

\section{Proof of Proposition \ref{keylem0}}

In this section we prove Proposition \ref{keylem0}. For its proof we need some basic facts which we recall next for the convenience of the reader; see, for example, \cite[1.2, 1.4 and 3.2]{M}.

\begin{chunk} \label{basic} Let $R$ be a ring, $x\in R$ and let $A$, $B$ and $C$ be $R$-modules. Set $\sigma=(0\to A \stackrel{f} \to B \stackrel{g} \to C \to 0) \in \Ext^1_R(C,A)$.
\begin{enumerate}[label=(\roman*)]
\item The connecting homomorphism $\Hom_R(C,C) \to \Ext^1_R(C,A)$ is given by the rule $\gamma \mapsto E$, where $E=(0\to A \to Z \to C \to 0)$ is the short exact sequence obtained by the following pull-back diagram:
\begin{align*}\notag{}
\vcenter{
\xymatrix{
0 \ar[r] & A \ar[r]^{f} & B \ar[r]^{g} \ar@{}[dr]|{\mathrm{PB}} & C \ar[r] & 0 \\
0 \ar[r] & A \ar[r] \ar@{=}[u] & Z \ar[r] \ar[u] & C \ar[r] \ar[u]_{\gamma} & 0 \\ 
}}
\end{align*}
\item The multiplication homomorphism $A \stackrel{x} \to A$ induces a homomorphism $\Ext^1_R(C,A)  \stackrel{x} \to \Ext^1_R(C,A)$ which sends $\sigma$ to $\sigma'$, where 
$\sigma'=(0\to A \to W \to C \to 0)$ is the short exact sequence obtained by the following push-out diagram:
\begin{align}\notag{}
\vcenter{
\xymatrix{
0 \ar[r] & A \ar[r]^{f} \ar[d]_{x} \ar@{}[dr]|{\mathrm{PO}} & B \ar[r]^{g} \ar[d] & C \ar[r] \ar@{=}[d] & 0 \\
0 \ar[r] & A \ar[r] \ar[d] & W \ar[r] \ar[d] & C \ar[r] & 0 \\
& A/xA \ar[d] \ar@{=}[r] & A/xA \ar[d] \\ 
& 0 & 0
}}
\end{align}
Therefore, it follows that $\sigma' \in x \cdot   \Ext^1_R(C,A)$.

Moreover, the diagram above induces the following commutative diagram where the leftmost square is a pushout square:
\begin{align}\notag{}
\vcenter{
\xymatrix{
0 \ar[r] & \Omega _R A \ar[r] \ar[d]_{x} \ar@{}[dr]|{\mathrm{PO}} & \Omega_R B \ar[r] \ar[d] & \Omega_R C \ar[r] \ar@{=}[d] & 0 \\
0 \ar[r] & \Omega_R A \ar[r]  & \Omega_R W \ar[r] & \Omega_R C \ar[r] & 0 
& 
& 
}}
\end{align}
Therefore, it follows that the bottom short exact sequence $0\to \Omega_R A \to \Omega_R W \to \Omega_R C \to 0$ belongs to $x \cdot   \Ext^1_R(\Omega_R C,\Omega_R A)$.
\item The multiplication homomorphism $C \stackrel{x} \to C$ induces a homomorphism $\Ext^1_R(C,A)  \stackrel{x} \to \Ext^1_R(C,A)$ which sends $\sigma$ to $\sigma''$, where 
$\sigma''=(0\to A \to V \to C \to 0)$ is the short exact sequence obtained by the following pull-back diagram:
\begin{align*}\notag{}
\vcenter{
\xymatrix{
& & 0 & 0 \\
& & C/xC \ar[u] \ar@{=}[r]& C/xC \ar[u]\\
0 \ar[r] & A \ar[r]^{f} & B \ar[r]^{g} \ar[u] \ar@{}[dr]|{\mathrm{PB}} & C \ar[r] \ar[u] & 0 \\
0 \ar[r] & A \ar[r] \ar@{=}[u] & V \ar[r] \ar[u] & C \ar[r] \ar[u]_x & 0 \\ 
}}
\end{align*}
Therefore, it follows that $\sigma'' \in x \cdot   \Ext^1_R(C,A)$.

\end{enumerate}
\end{chunk}

\begin{lem}\label{syz} Let $R$ be a ring, $x\in R$ and let $N$ be an $R$-module. Then the following are equivalent.
\begin{enumerate}[label=(\roman*)]
\item The multiplication map $N \stackrel{x}  {\longrightarrow} N$ factors through a free $R$-module.  
\item $x \cdot \Ext_R^i(N,-) = 0$ for each $i\geq 1$.
\item $x \cdot \Ext_R^1(N, \Omega_RN) = 0$.
\end{enumerate}

Furthermore, if one of these equivalent conditions holds and $x$ is a non zero-divisor on $N$, then there is an isomorphism $\Omega_R(N/xN) \cong N \oplus \Omega_RN$.
\end{lem}

\begin{proof} Note that the implication (ii) $\Rightarrow$ (iii) is trivial. Hence we show (i) $\Rightarrow$ (ii) and (iii) $\Rightarrow$ (i).  

To establish (i) $\Rightarrow$ (ii), we assume $N \stackrel{x}  {\longrightarrow} N$ factors through a free $R$-module $F$, i.e., there exist $R$-module homomorphisms $f$ and $g$ such that $N \stackrel{f}  {\longrightarrow} F \stackrel{g}  {\longrightarrow} N$, where $gf=x \cdot 1_N$. Now let $X$ be an $R$-module and $n\geq 1$ be an integer. Then $f$ and $g$ induce $R$-module homomorphisms $f^{\ast}$ and $g^{\ast}$ such that $\Ext^n_R(N,X) \stackrel{g^{\ast}}  {\longrightarrow} \Ext^n_R(F,X) \stackrel{f^{\ast}} {\longrightarrow}\Ext^n_R(N,X)$, where $f^{\ast}g^{\ast}=x \cdot 1_{\Ext^n_R(N,X)}$. As  $\Ext^n_R(F,X)$ vanishes, we conclude that $f^{\ast}g^{\ast}=0$, i.e., $x \cdot \Ext^n_R(N,X)=0$. This proves the implication (i) $\Rightarrow$ (ii).

Next consider the syzygy exact sequence $E=(0 \to \Omega_R N \to G \xrightarrow{p} N \to 0)$, where $G$ is free. This induces the exact sequence
$0 \to \Hom_R(N, \Omega_R N) \to \Hom_R(N, G) \xrightarrow{p_*} \Hom_R(N, N) \to \Ext_R^1(N, \Omega_R N)$. Note that $1_N \mapsto E$ under the connecting homomorphism $\Hom_R(N, N) \to \Ext_R^1(N, \Omega_R N)$; see \ref{basic}(i). So the image of the map $N \stackrel{x}  {\longrightarrow} N$ under the connecting homomorphism belongs to $x \cdot \Ext_R^1(N, \Omega_RN)$. 

Now assume $x \cdot \Ext_R^1(N, \Omega_RN)=0$. Then the multiplication map $N \stackrel{x}  {\longrightarrow} N$ is in $\im(p_*)$, and hence it factors through the free module $G$. Consequently, (iii) $\Rightarrow$ (i) follows.

Next assume $x$ is a non zero-divisor on $N$. Then we consider the multiplication map $N \stackrel{x} \to N$ and make use of  \ref{basic}(iii) with the exact sequence $E$, and obtain short exact sequences of $R$-modules:
$$E_1=(0 \to V \to G \to N/xN \to 0) \text{ and } E_2=(0 \to \Omega_R N \to V \to N \to 0) \in x \cdot \Ext_R^1(N, \Omega_RN)=0.$$ Now $E_2$ splits so that $E_1$ yields the isomorphism $\Omega_R(N/xN) \cong V \cong N \oplus \Omega_R N$, as required.
\end{proof}

Next we use Lemma \ref{syz} and give a proof of Proposition \ref{keylem0}. We also need the following fact:

\begin{chunk} \label{use} Let $R$ be a local ring and let $0 \to A \to B \to C \to 0$ be a short exact sequence of $R$-modules. Then there is a short exact sequence $0 \to \Omega_RC \to A \oplus H \to B \to 0$, where $H$ is a free $R$-module; see, for example, \cite[2.2]{DTgrade}. Therefore, if $A$ is free, then $\Omega_R C \cong \Omega_R B$.
\end{chunk}

\begin{proof}[Proof of Proposition \ref{keylem0}] Note that, since $\underline{x}$ is $R$-regular and $M$ is an $n$th syzygy module, we see that $\underline{x}$ is also $M$-regular; see \ref{tf}(ii).
We proceed by induction on $n$. First assume $n=1$.

As in the proof of Lemma \ref{syz}, we look at the syzygy exact sequence $E=(0 \to \Omega_R N \to F \to N \to 0)$, where $F$ is free. Then, by using the multiplication map $M \stackrel{x_1} \to M$ and \ref{basic}(ii), we obtain short exact sequences of $R$-modules of the form
$$E_1=(0 \to F \to W \to M/xM \to 0) \text{ and } E_2=(0 \to \Omega_R N \to W \to N \to 0) \in x_1 \cdot \Ext_R^1(N, \Omega_RN)=0.$$ Now $E_2$ splits, and hence $E_1$ yields the required short exact sequence.


Next we assume $n > 1$, and set $N' = N \oplus \Omega_R N $, $M' = \Omega_R^{n-1} N' \cong \Omega_R^{n-1}N \oplus \Omega_R^n N$, and $\underline{x'}=x_1, \ldots, x_{n-1}$.
Note that it follows:
\begin{equation}\tag{\ref{keylem0}.2}
\Ext_R^1(N', \Omega_R N') = \Ext_R^1(N, \Omega_R N ) \oplus  \Ext_R^1(N, \Omega^2_R N ) \oplus \Ext_R^2(N, \Omega_R N ) \oplus  \Ext_R^2(N, \Omega^2_R N).
\end{equation}
As  $\underline{x} \cdot \Ext_R^1(N, \Omega_R N) = 0$, we see from Lemma \ref{syz} that $\underline{x} \cdot \Ext_R^i(N, -) = 0$ for all $i\geq 1$. Therefore, by (\ref{keylem0}.2), we conclude that $\underline{x}$, and hence $\underline{x'}$ annihilates the module $\Ext_R^1(N', \Omega_R N')$. Thus the following short exact sequence exists due to the induction hypothesis:
\begin{align}\tag{\ref{keylem0}.3}
0 \to F' \to \bigoplus_{i = 0}^{n-1} \Omega_R^{i+n-2}(N')^{\oplus \left( \begin{smallmatrix} n-1 \\ i \\ \end{smallmatrix} \right)} \to \Omega_R^{n-2}(M'/\underline{x'}M') \to 0,
\end{align}
where $F'$ is a free $R$-module. Furthermore, as $M' = \Omega_R^{n-1} N'$, we use \ref{use} along with (\ref{keylem0}.3) and obtain:
\begin{align}\tag{\ref{keylem0}.4}
\Omega_R^{n-1}(M'/\underline{x'}M') \cong \bigoplus_{i = 0}^{n-1} \Omega_R^{i} (M')^{\oplus \left( \begin{smallmatrix} n-1 \\ i \\ \end{smallmatrix} \right)}
\end{align}

Recall that $M=\Omega ^n_R N$. Hence there is a short exact sequence $0 \to M \to F \to \Omega_R^{n-1}N \to 0$ for some free $R$-module $F$. 
It follows, since $\underline{x'}$ is $R$-regular, that $\underline{x'}$ is $\Omega_R^{n-1}N$-regular; see \ref{tf}(ii). So we have a short exact sequence of the form:
$$
0 \to M/\underline{x'}M  \stackrel{\alpha}
     {\longrightarrow} F/\underline{x'} F \to \Omega_R^{n-1}N /\underline{x'}\Omega_R^{n-1} N \to 0.
$$
We take the pushout of $\alpha$ and the injective map $M/\underline{x'}M \stackrel{x_n}
     {\longrightarrow} M/\underline{x'}M$, and obtain the following commutative diagram:
\begin{align}\tag{\ref{keylem0}.5}
\begin{aligned}
\xymatrix{
& 0 \ar[d] & 0 \ar[d]\\
0 \ar[r] & M/\underline{x'}M \ar[r]^{\alpha} \ar[d]_{x_n} \ar@{}[dr]|{\mathrm{PO}} & F/\underline{x'} F \ar[r] \ar[d] & \Omega_R^{n-1}(N)/\underline{x'}\Omega_R^{n-1}(N) \ar[r] \ar@{=}[d] & 0 \\
0 \ar[r] & M/\underline{x'}M \ar[r] \ar[d] & W \ar[r] \ar[d] & \Omega_R^{n-1}(N)/\underline{x'}\Omega_R^{n-1}(N) \ar[r] & 0 \\
& M/\underline{x}M \ar[d] \ar@{=}[r] & M/\underline{x}M \ar[d] \\ 
& 0 & 0
}
\end{aligned}
\end{align}
Note that the short exact sequence $0 \to \Omega_R^{n-1}(M/\underline{x'}M) \to \Omega_R^{n-1}W \to \Omega_R^{n-1}\big(\Omega_R^{n-1}N/\underline{x'}\Omega_R^{n-1}N\big) \to 0$ belongs to $x_n \cdot \Ext^1(\Omega_R^{n-1}\big(\Omega_R^{n-1} N /\underline{x'}\Omega_R^{n-1}N \big), \Omega_R^{n-1}(M/\underline{x'}M))$; see (\ref{keylem0}.5) and \ref{basic}(ii).

Next note that we have the following isomorphisms:
\begin{equation}\tag{\ref{keylem0}.6}
\Ext_R^1(\Omega_R^{n-1}(M'/\underline{x'}M'), -) \cong \bigoplus_{i=0}^{n-1}\Ext_R^{i+1}(M', -)^{\oplus \left( \begin{smallmatrix} n-1 \\ i \\ \end{smallmatrix} \right)} \cong \bigoplus_{i=n}^{2n}\Ext_R^i(N, -)^{\oplus r(i)},
\end{equation}
where $r(i)$ is a positive integer depending on $i$. The first isomorphism in (\ref{keylem0}.6) is due to (\ref{keylem0}.4), while the second one follows from the fact that $M'\cong \Omega_R^{n-1}N \oplus \Omega_R^n N$.

Recall that $\Omega_R^{n-1} N$ is a direct summand of $M'$. Therefore, $\Omega_R^{n-1}\big(\Omega_R^{n-1} N /\underline{x'}\Omega_R^{n-1}N \big)$ is a direct summand of $\Omega_R^{n-1}(M'/\underline{x'}M')$. This implies, in view of (\ref{keylem0}.6), that $\Ext^1(\Omega_R^{n-1}\big(\Omega_R^{n-1} N /\underline{x'}\Omega_R^{n-1}N \big),-)$ is a direct summand of $\bigoplus_{i=n}^{2n}\Ext_R^i(N, -)^{\oplus r(i)}$. It follows, since $\underline{x} \cdot \Ext_R^i(N, -) = 0$ for all $i\geq 1$, that $x_n$ annihilates each direct summand of $\Ext_R^i(N, -)$ for each $i\geq 1$; in particular, we conclude that $x_n \cdot \Ext_R^1(\Omega_R^{n-1}\big(\Omega_R^{n-1} N /\underline{x'}\Omega_R^{n-1}N \big), \Omega_R^{n-1}(M/\underline{x'}M))=0$. This implies that the bottom short exact sequence in (\ref{keylem0}.5) splits so that we have the following isomorphism:
\begin{equation}\tag{\ref{keylem0}.7}
\Omega_R^{n-1} W \cong \Omega_R^{n-1}(M/\underline{x'}M) \oplus \Omega_R^{n-1}\big(\Omega_R^{n-1}N/\underline{x'}\Omega_R^{n-1}N\big). 
\end{equation}

Recall that, by (\ref{keylem0}.5), we have a short exact sequence $0 \to F/\underline{x'} F \to W \to M/\underline{x}M  \to 0$. Hence, by taking syzygy and using (\ref{keylem0}.7), we obtain the exact sequence:
\begin{equation}\tag{\ref{keylem0}.8}
0 \to \Omega_R^{n-1}(F/\underline{x'}F) \to \Omega_R^{n-1}(M/\underline{x'}M) \oplus \Omega_R^{n-1}(\Omega_R^{n-1}N/\underline{x'}\Omega_R^{n-1}N) \to \Omega_R^{n-1}(M/\underline{x}M) \to 0.
\end{equation}

The minimal free resolution $F_{\bullet}$ of $F/\underline{x'}F$ is of the form $0 \to F \to F^{\oplus n-1} \to \cdots \to F^{\oplus n-1} \to F \to 0$ since $\HH_i(F_{\bullet} \otimes_R \K(\underline{x'};R))=\Tor_i^R(F, R/\underline{x'}R)=0$ for all $i\geq 0$, where $\K(\underline{x'};R)$ is the Koszul complex of $R$ with respect to $\underline{x'}$. Therefore, it follows that:
\begin{equation}\tag{\ref{keylem0}.9}
\Omega_R^{n-1}(F/\underline{x'}F) \cong F.
\end{equation}

We have the following isomorphisms about the middle module in the short exact sequence (\ref{keylem0}.8):
\begin{align*}
\Omega_R^{n-1}(M/\underline{x'}M) \oplus \Omega_R^{n-1}\big(\Omega_R^{n-1}N/\underline{x'}\Omega_R^{n-1}N\big) &\cong \Omega_R^{n-1}(M'/\ul{x'}M') \\
&\cong \bigoplus_{i = 0}^{n-1} \Omega_R^{i} (M')^{\oplus \left( \begin{smallmatrix} n-1 \\ i \\ \end{smallmatrix} \right)} \\ 
&\cong \bigoplus_{i = 0}^{n-1} \big( \Omega_R^{i+n-1}N \oplus \Omega^{i+n}_RN \big)^{\oplus \left( \begin{smallmatrix} n-1 \\ i \\ \end{smallmatrix} \right)} \\ \tag{\ref{keylem0}.10}
&\cong
\Bigg[ \bigoplus_{i = 0}^{n-1} \big( \Omega_R^{i+n-1} N \big)^{\oplus \left( \begin{smallmatrix} n-1 \\ i \\ \end{smallmatrix} \right)} \Bigg]
\bigoplus 
\Bigg[\bigoplus_{i = 1}^{n} \big(\Omega_R^{i+n-1} N \big)^{\oplus \left( \begin{smallmatrix} n-1 \\ i-1 \\ \end{smallmatrix} \right)} \Bigg] \\
&\cong \Bigg[ \big(\Omega_R^{n-1} N \big)^{\oplus \left( \begin{smallmatrix} n \\ 0 \\ \end{smallmatrix} \right)} \Bigg]
\bigoplus 
\Bigg[\bigoplus_{i = 1}^{n-1} \big(\Omega_R^{i+n-1} N \big)^{\oplus \left( \begin{smallmatrix} n \\ i \\ \end{smallmatrix} \right)} \Bigg]
\bigoplus 
\Bigg[ \big(\Omega_R^{2n-1} N \big)^{\oplus \left( \begin{smallmatrix} n \\ n \\ \end{smallmatrix} \right)} \Bigg]\\
&\cong \bigoplus_{i = 0}^{n} \big(\Omega_R^{i+n-1}N \big)^{\oplus \left( \begin{smallmatrix} n \\ i \\ \end{smallmatrix} \right)}.
\end{align*}
In (\ref{keylem0}.10), the first and the third isomorphisms follow since $M' \cong \Omega^{n-1}_RN \oplus \Omega^{n}_RN = \Omega^{n-1}_RN \oplus M$, while the second isomorphism is nothing but (\ref{keylem0}.4). 
The other isomorphisms are elementary.

Now, in view of (\ref{keylem0}.9) and (\ref{keylem0}.10), we conclude that the short exact sequence in (\ref{keylem0}.8) is the required one. This completes the induction argument and hence the proof of the proposition.
\end{proof}



We end this section with a consequence of Proposition \ref{keylem0} which corroborates \cite[2.1]{HP} and \cite[2.2]{Tak10}.


\begin{cor} \label{propint} Let $R$ be a local ring, $N$ a nonzero $R$-module, and let $M= \Omega^n_R N$ for some $n\geq 1$.
Assume there is an $R$-regular sequence $\underline{x}= x_1, \ldots, x_n$ of length $n$ such that $\underline{x} \cdot \Ext_R^1(N, \Omega_R N) = 0$.
Then the following isomorphism holds:
\begin{equation}\tag{\ref{propint}.1}
\Omega_R^n(M/\underline{x}M) \cong \bigoplus_{i = 0}^{n} \Omega_R^{i}(M)^{\oplus \left( \begin{smallmatrix} n \\ i \\ \end{smallmatrix} \right)}
\end{equation}
\end{cor}

\begin{proof} It follows from Proposition \ref{keylem0} that we have the following short exact sequence:
\begin{equation}\notag{}
0 \longrightarrow F \longrightarrow \bigoplus_{i = 0}^n \bigg(\Omega_R^{i+n-1}N\bigg)^{\oplus\left(\begin{smallmatrix} n \\ i \\ \end{smallmatrix} \right)} \longrightarrow\Omega_R^{n-1}(M/\underline{x}M) \longrightarrow 0, 
\end{equation}
where $M=\Omega_R ^nN$. Therefore \ref{use} yields the short exact sequence
\begin{equation}\notag{}
0 \longrightarrow \Omega_R \big( \Omega_R^{n-1}(M/\underline{x}M) \big)  \longrightarrow F \oplus G  \longrightarrow  \bigoplus_{i = 0}^n \bigg(\Omega_R^{i+n-1}N\bigg)^{\oplus\left(\begin{smallmatrix} n \\ i \\ \end{smallmatrix} \right)}\longrightarrow 0, 
\end{equation}
where $G$ is a free $R$-module. Hence, we conclude that:
\begin{equation}\notag{}
\Omega_R^n(M/\underline{x}M) \cong \Omega_R 
\Bigg( 
\bigoplus_{i = 0}^n \Bigg(\Omega_R^{i+n-1}N\bigg)
\Bigg) ^{\oplus\left(\begin{smallmatrix} n \\ i \\ \end{smallmatrix} \right)} \cong \bigoplus_{i = 0}^n \Bigg(\Omega_R^{i+n}N\bigg)^{\oplus\left(\begin{smallmatrix} n \\ i \\ \end{smallmatrix} \right)} \cong  \bigoplus_{i = 0}^{n} \Omega_R^{i}(M)^{\oplus \left( \begin{smallmatrix} n \\ i \\ \end{smallmatrix} \right)}.
\end{equation}
\end{proof}


\appendix
\section{On Tor-rigid modules over complete intersection rings}\label{app}


Recall that, if $R$ is a hypersurface ring, that is quotient of an unramified regular local ring, then each $R$-module that has finite projective dimension is Tor-rigid; see \ref{TR}(ii). In this section we generalize this result and observe that modules that are eventually periodic of period one are Tor-rigid over such hypersurfaces. In fact, we show that such periodic modules are $c$-Tor-rigid over complete intersections of codimension $c$; see \ref{A2}. In particular, we conclude that modules that are eventually periodic of period one satisfy the depth inequality of Theorem \ref{thmint}; see \ref{A3}.

Throughout, $R$ denotes a local complete intersection ring such that $\widehat{R}=S/(\underline{x})$ for some unramified regular ring $(S, \fn)$ and some $S$-regular sequence $\underline{x}\subseteq \fn^2$ of length $c$, where $c\geq 1$. The main tool we use in this section is the \emph{eta function} of Dao, which we recall next.

\begin{dfn}\label{function} (\cite[4.2, 4.3(1), 5.4]{Da2}; see also \cite[3.3]{CeD}) Let $M$ and $N$ be $R$-modules. Assume we have $\length_R(\Tor_i^R(M,N))<\infty$ for all $i\gg 0$. Set $f=\inf\{s:\len_{R}(\Tor_i^R(M,N))<\infty \textnormal{ for all } i\geq s \}$. Then the eta function $\h^R(M,N)$ is defined as follows:
\begin{equation}
\begin{split}
 \h^R(M,N) = \lim_{n\to\infty} \frac{\displaystyle{\sum\limits^{n}_{i=f}(-1)^i \length_R(\Tor_i^R(M,N))}}{\displaystyle{n^c}}
\end{split}\notag{}
\end{equation} \pushQED{\qed} 
\qedhere 
\popQED
\end{dfn}




In the following we collect some properties of the eta function:

\begin{chunk} \label{Te} Let $M$ and $N$ be $R$-modules.
\begin{enumerate}[label=(\roman*)]
\item If $\h^R(M, N)=0$, then the pair $(M,N)$ is $c$-Tor-rigid; see \ref{TR} and \cite[6.3]{Da2}. For example, if $c=1$ and $R$ is a simple hypersurface singularity of even dimension, then it follows that $\h^R(M, N)=0$ for all $R$-modules $M$ and $N$ so that each module is Tor-rigid over $R$; see \cite[4.4]{Da2} and \cite[3.16]{Da1}.
\item The eta function is additive whenever it is defined. Namely, if $0 \to M' \to M \to M'' \to 0$ is a short exact sequence of $R$-modules such that $\Tor_i^R(M', N)$ and $\Tor_i^R(M'', N)$ have finite length for all $i \gg 0$, then it follows that $\h^R(M, N) = \h^R(M', N) + \h^R(M'', N)$; see \cite[4.3(2)]{Da2}. \pushQED{\qed} 
\qedhere 
\popQED
\end{enumerate}
\end{chunk}

We proceed to observe that modules that are eventually periodic of period one are $c$-Tor-rigid over $R$.

\begin{chunk} \label{A1} Let $N$ be an $R$-module such that $N$ is eventually periodic of period one, i.e., $\Omega_R ^n N \cong \Omega_R^{n+1} N$ for all $n\gg 0$. If $X$ is an $R$-module and $\Tor_i^R(N,X)$ has finite length for all $i\gg 0$, then the pair $(N,X)$ is $c$-Tor-rigid over $R$.

To see this, first note that $\h^{R}(N,X)$ is well-defined; see \ref{Te}. Moreover, for $n\gg 0$, the following equalities hold:
\begin{align*}
\h^{R}(N,X) &= (-1)^n \h^{R}(\Omega_R^n N, X) \\
&= (-1)^n \h^{R}( \Omega_R^{n+1} N, X)  \\
&= (-1)^n (-1)^{n+1} \h^{R}( N, X)  \\
&= - \h^{R}(N, X).  
\end{align*}
Here, the first and third equalities are due to \ref{Te}(ii), while the second one follows by the hypothesis. Consequently, we conclude $\h^{R}(N, X)=0$, and this implies that the pair $(N,X)$ is $c$-Tor-rigid; see \ref{Te}(i).
\end{chunk}

\begin{chunk} \label{A2} Let $N$ be an $R$-module such that $\Omega_R ^n N \cong \Omega_R^{n+1} N$ for all $n\gg 0$.  Then it follows that $N$ is $c$-Tor-rigid.

To see this, let $X$ be an $R$-module with $\Tor_1^R(N,X)=\ldots = \Tor_c^R(N,X)=0$. We set $r=\dim_R(N\otimes_RX)$ and proceed by induction on $r$ to show that $\Tor_i^R(N,X)=0$ for all $i\geq 1$. 

If $r\leq 0$, then the claim follows from \ref{A1}. So we assume $r\geq 1$, and pick $\fp \in \Supp_R(N\otimes_RX)$ such that $\fp\neq \fm$. Note that $\Omega_{R_{\fp}} ^n N_{\fp} \cong \Omega_{R_{\fp}}^{n+1} N_{\fp}$ for all $n\gg 0$. Then it follows by the induction hypothesis that $\Tor_i^R(N,X)_{\fp}=0$ for all $i\geq 1$. This shows that $\Tor_i^R(N,X)$ has finite length for all $i\geq 1$. Hence, by \ref{A1}, the pair $(N, X)$ is $c$-Tor-rigid over $R$. Thus, as $\Tor_1^R(N,X)=\ldots = \Tor_c^R(N,X)=0$, we conclude that $\Tor_i^R(N,X)$ vanishes for each $i\geq 1$, as claimed. 
\end{chunk}

\begin{chunk}  \label{A3} Let $R$ be a hypersurface ring, $\fa$ be an ideal of $R$, and let $N$ be an $R$-module.
\begin{enumerate}[label=(\roman*)]
\item If $N$ is an $R$-module such that $\Omega_R ^n N \cong \Omega_R^{n+1} N$ for all $n\gg 0$, then it follows that $N$ is Tor-rigid and hence $\depth_R(\fa, N) \leq \depth_R(\fa, R)$; see Theorem \ref{thmint} and \ref{A2}.
\item If $\Omega_R N \cong M \oplus \Omega_R M$ for some $R$-module $M$, then it follows from part (i) that $N$ is Tor-rigid over $R$ and hence $\depth_R(\fa, N) \leq \depth_R(\fa, R)$: this is because $M$ is eventually periodic of period at most two \cite{Ei} and hence $N$ is eventually periodic of period one.\pushQED{\qed} 
\qedhere 
\popQED
\end{enumerate}
\end{chunk}

If $R$ is hypersurface, then it is clear that modules of the form $M \oplus \Omega_R M$ are Tor-rigid over $R$; see \ref{TR}(ii). On the other hand, the fact that modules as in \ref{A3}(ii) are Tor-rigid over $R$ seems interesting to us since a module over a hypersurface ring need not be Tor-rigid in general, even if its syzygy is Tor-rigid.

\section*{acknowlwdgements} The authors are grateful to  Shunsuke Takagi for his help and explaining the arguments of \ref{Tak1} and \ref{Tak2} to them.

\bibliography{a}
\bibliographystyle{amsplain}

\end{document}